\documentclass[12pt]{amsart}

\usepackage{amsmath,amsthm,amsfonts,amssymb}
\usepackage{times,latexsym,mathabx,url,tikz}
\usepackage{graphicx}

\newtheorem{theorem}{Theorem}[section]
\newtheorem{prop}[theorem]{Proposition}
\newtheorem{lem}[theorem]{Lemma}

\newtheorem{rem}[theorem]{Remark}
\newtheorem{ex}[theorem]{Example}
\numberwithin{equation}{section}

\renewcommand{\o}{\omega}

\newcommand{\N}{\mathbb{N}}

\renewcommand{\leq}{\leqslant}
\renewcommand{\geq}{\geqslant}

\makeatletter
\renewcommand{\pmod}[1]{\allowbreak\mkern7mu({\operator@font mod}\,\,#1)}
\makeatother

\newcommand{\be}{\begin{equation}}
\newcommand{\ee}{\end{equation}}

\renewcommand{\le}{\leqslant}

\title{Tree asymptotic densities in number theory}

\author{Roberto Conti}
\address{Department of Basic and Applied Sciences for Engineering,
Sapienza University of Rome, Italy}
\email{roberto.conti@sbai.uniroma1.it}

\author{Pierluigi Contucci}
\address{Department of Mathematics, University of Bologna, Italy}
\email{pierluigi.contucci@unibo.it}

\author{Vitalii Iudelevich}
\address{Faculty of Computer Science, National Research University
``Higher School of Economics'', Pokrovsky blvd.\ 11, Moscow 109028,
Russia}
\email{vitaliiiudelevich@gmail.com}

\date{\today}

\begin{document}

\begin{abstract}
We study the asymptotic distribution of integers sharing the same
rooted-tree structure that encodes their complete prime factorization tower. For each tree we derive an explicit density formula depending only on a pair $(m,k)$, the density signature of the tree, up to a suitable multiplicative scalar factor and introduce the corresponding tree zeta function, which generalizes the prime zeta function. Classical results such as the prime number theorem and later work by Landau 
appear as special cases.
\end{abstract}

\maketitle


\section{Introduction}
The Prime Number Theorem has a long and interesting history,
 going back to work of Legendre, Gauss, Riemann, Hadamard and de la Vall\'ee Poussin.
It is beyond any doubt one of the most fascinating results in number theory, and as such it continues to inspire ideas and results ever since. 
It provides the asymptotic distribution of the primes, 
i.e. the asymptotic behaviour of the prime-counting function 
$\pi(x) = \#\{ p \leq x \ | \ p \ \mbox{prime}\}$ for $x \to +\infty$. 
More precisely, it states that
$$
\pi(x) \sim \frac{x}{\log x}, \quad x \to +\infty \ .
$$

A natural generalization was then provided by
Landau, dealing with numbers that can be written as products of $k$ distinct primes, for any fixed $k > 1$. 
Define $\omega(n)$ to be the number of distinct prime divisors of $n$. Then (see \cite{Lan})
$$\#\left\{n\leqslant x: n\ \text{is square-free and}\ \omega(n) = k \right\} \sim \dfrac{x (\log\log x)^{k-1}}{(k-1)! \log x}$$
for any fixed $k\geqslant 2$ and $x\to +\infty$.
(There are variants of this result, e.g. for the number of the integers $n\leqslant x$ with $\omega(n) = k$, that actually display the same asymptotic behaviour.)

Quite recently, Naslund studied the asymptotics of the number of integers with a ``predetermined prime factorization'' \cite{Nas}. Namely, given $(\alpha_1,\alpha_2,\ldots,\alpha_k) \in {\mathbb N}^k$, he described the asymptotics of the number of integers from 1 up to $x$ of the form 
$$n = p_1^{\alpha_1} p_2^{\alpha_2} \cdots p_k^{\alpha_k} \ , $$ where the $p_i$ are (not necessarily) distinct.
For instance, he showed that 
$$\#\{n \leq x \ : \ n = p q^3, p, q \ \text{primes}, \ p \neq q\} \sim \frac x {\log x} P(3) \ , \ x \to + \infty$$
as well as
$$\#\{n \leq x \ : \ n = p_1 p_2 p_3^3 p_4^5 p_5^{19}, p_i \ \text{distinct primes}\} \sim \frac {x \log\log x} {\log x} P(3)P(5)P(19) \ , \ x \to + \infty \ , $$
where $P(s) = \sum_p p^{-s}$ is the so-called {\it prime zeta function} (see e.g. \cite{Fr,Tit}).

In the present paper we push further this line of thought and consider the (non-planar) rooted tree structure of natural numbers (which is the graphical encoding of the so-called
prime factorization tower and can be inferred from that) 
as recently discussed in \cite{CC25,Iu22,Iu25,CCI24}.
In other words, we replace the above ``predetermined prime factorization'' with the
``predetermined prime factorization tower''.
Indeed, our main result is a formula for the asymptotic distribution of integers with a given
rooted tree structure.

\medskip
Hereafter we describe in more detail our result, along with the necessary concepts. 
To any natural number $n$ one can associate in unique way a nonplanar rooted tree $t(n)$ describing iteratively the prime number factorization of $n$ and then the factorization of all the integers appearing as exponents in this factorization, and then the factorization of all the integers appearing as exponents in the factorization of the exponents, and so on and so fourth.
The writing 
$$n = \prod_{i=1}^{h} p_i^{\prod_{j=1}^{h_i} p_{i,j}^{\prod_{k=1}^{h_{i,j}} p_{i,j,k}^{(\prod\cdots)}}}$$
is sometimes called the prime tower factorization of $n$ (see \cite{DKV} and \cite{DG}).
The tree $t(n)$ can be read off from this tower.
More concretely, let
$p_1^{\alpha_1} \cdots p_k^{\alpha_k}$ be the prime factorization of $n$. 
Correspondingly, we draw
$k$ edges emanating from the root. 
Then consider the prime factorization of $\alpha_1$, say
$q_1^{\beta_1} \cdots q_h^{\beta_h}$.
We then draw $h$ edges emanating from the end of the first edge attached to the root, and repeat this construction both horizontally (i.e. factorizing the other exponents $\alpha_2, \ldots, \alpha_k$, and drawing edges emanating from the corresponding ends) and vertically
(i.e. factorizing $\beta_1, \ldots, \beta_h$ and drawing edges on the top of the previous ends, etc.),
climbing the ladder of exponents, until there is nothing more to factorize. 
Thus, for example, 

\begin{equation}
t(16)=t(2^{2^{2}})= \; 
\begin{tikzpicture}[baseline={(r.base)},scale=0.8,level distance=8mm]
\tikzset{
  root/.style={circle,draw,fill=black,inner sep=1pt},
  hollow/.style={circle,draw,fill=white,inner sep=1pt}
}
\node[root] (r) {}
  child[grow=90]{ node[hollow] (a) {}
    child[grow=90]{ node[hollow] (b) {}
      child[grow=90]{ node[hollow] {} } }};
\end{tikzpicture}
\end{equation}
\begin{equation}
t(300)=t(2^{2}\cdot3\cdot5^{2})=
\begin{tikzpicture}[baseline={(root.south)},scale=0.8]
  \tikzset{
    root/.style={circle,draw,fill=black,inner sep=1pt},
    hollow/.style={circle,draw,fill=white,inner sep=1pt}
  }

  \node[root] (root) at (0,0) {};

  \node[hollow] (L1) at (-0.9,0.9) {}; 
  \node[hollow] (M1) at ( 0.0,0.9) {}; 
  \node[hollow] (R1) at ( 0.9,0.9) {}; 

  \draw (root)--(L1);
  \draw (root)--(M1);
  \draw (root)--(R1);

  \node[hollow] (L2) at (-0.9,1.8) {};
  \draw (L1)--(L2);

  \node[hollow] (R2) at ( 0.9,1.8) {};
  \draw (R1)--(R2);
\end{tikzpicture}
\end{equation}

\begin{equation}
\label{eq:t4800}
t(4\,800)=t(2^{6}\cdot 3\cdot 5^{2})
= t(2^{2\cdot 3}\cdot 3\cdot 5^{2})=
\begin{tikzpicture}[scale=0.85, baseline={(root.south)}]
  \tikzset{
    root/.style={circle, fill=black, draw=black, inner sep=0, minimum size=3pt},
    v/.style={circle, draw=black, fill=white, inner sep=0, minimum size=3pt}
  }

  \node[root] (root) at (0,0) {};

  \node[v] (L1) at (-0.80,0.90) {};
  \node[v] (M1) at ( 0.00,0.90) {};
  \node[v] (R1) at ( 0.80,0.90) {};

  \node[v] (L2a) at (-1.15,1.80) {};
  \node[v] (L2b) at (-0.45,1.80) {};
  \node[v] (R2)  at ( 0.80,1.80) {};

  \draw (root)--(L1);
  \draw (root)--(M1);
  \draw (root)--(R1);

  \draw (L1)--(L2a);
  \draw (L1)--(L2b);

  \draw (R1)--(R2);
\end{tikzpicture}
\end{equation}

\begin{equation}
\label{eq:t307200}
t(307\,200)=t(2^{12}\cdot3\cdot5^{2})
= t(2^{2^{2}\cdot3}\cdot3\cdot5^{2})=
\begin{tikzpicture}[scale=0.85, baseline={(root.south)}]
  \tikzset{
    root/.style={circle, fill=black, draw=black, inner sep=0, minimum size=3pt},
    v/.style={circle, draw=black, fill=white, inner sep=0, minimum size=3pt}
  }

  \node[root] (root) at (0,0) {};

  \node[v] (L1) at (-0.8,0.9) {};
  \node[v] (M1) at ( 0.0,0.9) {};
  \node[v] (R1) at ( 0.8,0.9) {};
  \draw (root)--(L1);
  \draw (root)--(M1);
  \draw (root)--(R1);

  \node[v] (L2a) at (-1.1,1.8) {}; 
  \node[v] (L2b) at (-0.5,1.8) {}; 
  \draw (L1)--(L2a);
  \draw (L1)--(L2b);

  \node[v] (L3) at (-1.1,2.7) {};
  \draw (L2a)--(L3);

  \node[v] (R2) at (0.8,1.8) {};
  \draw (R1)--(R2);
\end{tikzpicture}
\end{equation}

\begin{equation}
\label{eq:t18662400}
t(18\,662\,400)=t(2^{10}\cdot3^{6}\cdot5^{2})
= t(2^{2\cdot5}\cdot3^{2\cdot3}\cdot5^{2})=
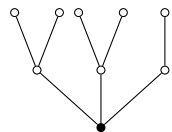
\begin{tikzpicture}[scale=0.85, baseline={(root.south)}]
  \tikzset{
    root/.style={circle, fill=black, draw=black, inner sep=0, minimum size=3pt},
    v/.style={circle, draw=black, fill=white, inner sep=0, minimum size=3pt}
  }

  \node[root] (root) at (0,0) {};

  \node[v] (L1) at (-1.00,0.90) {}; 
  \node[v] (M1) at ( 0.00,0.90) {}; 
  \node[v] (R1) at ( 1.00,0.90) {}; 

  \draw (root)--(L1);
  \draw (root)--(M1);
  \draw (root)--(R1);

  \node[v] (L2a) at (-1.35,1.80) {}; 
  \node[v] (L2b) at (-0.65,1.80) {}; 
  \draw (L1)--(L2a);
  \draw (L1)--(L2b);

  \node[v] (M2a) at (-0.35,1.80) {};
  \node[v] (M2b) at ( 0.35,1.80) {};
  \draw (M1)--(M2a);
  \draw (M1)--(M2b);

  \node[v] (R2) at ( 1.00,1.80) {};
  \draw (R1)--(R2);
\end{tikzpicture}
\end{equation}

\begin{equation}
\label{eq:t192000000}
t(192\,000\,000)=t(2^{12}\cdot 3 \cdot 5^{6})
= t(2^{2^{2}\cdot 3}\cdot 3 \cdot 5^{2\cdot 3})=
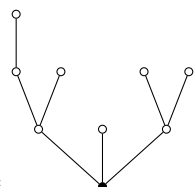
\begin{tikzpicture}[scale=0.85, baseline={(root.south)}]
  \tikzset{
    root/.style={circle, fill=black, draw=black, inner sep=0, minimum size=3pt},
    v/.style={circle, draw=black, fill=white, inner sep=0, minimum size=3pt}
  }

  \node[root] (root) at (0,0) {};

  \node[v] (L1) at (-1.00,0.90) {}; 
  \node[v] (M1) at ( 0.00,0.90) {}; 
  \node[v] (R1) at ( 1.00,0.90) {}; 

  \draw (root)--(L1);
  \draw (root)--(M1);
  \draw (root)--(R1);

  \node[v] (L2a) at (-1.35,1.80) {}; 
  \node[v] (L2b) at (-0.65,1.80) {}; 
  \draw (L1)--(L2a);
  \draw (L1)--(L2b);

  \node[v] (L3)  at (-1.35,2.70) {};
  \draw (L2a)--(L3);


  \node[v] (R2a) at ( 0.65,1.80) {}; 
  \node[v] (R2b) at ( 1.35,1.80) {}; 
  \draw (R1)--(R2a);
  \draw (R1)--(R2b);
\end{tikzpicture}
\end{equation}

\begin{equation}
\label{eq:t729000000}
t(729\,000\,000)=t(2^{6}\cdot 3^{6}\cdot 5^{6})
= t(2^{2\cdot 3}\cdot 3^{2\cdot 3}\cdot 5^{2\cdot 3})=
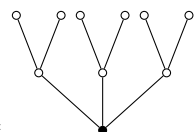
\begin{tikzpicture}[scale=0.85, baseline={(root.south)}]
  \tikzset{
    root/.style={circle, fill=black, draw=black, inner sep=0, minimum size=3pt},
    v/.style={circle, draw=black, fill=white, inner sep=0, minimum size=3pt}
  }

  \node[root] (root) at (0,0) {};

  \node[v] (L1) at (-1.00,0.90) {}; 
  \node[v] (M1) at ( 0.00,0.90) {}; 
  \node[v] (R1) at ( 1.00,0.90) {}; 

  \draw (root)--(L1);
  \draw (root)--(M1);
  \draw (root)--(R1);

  \node[v] (L2a) at (-1.35,1.80) {};
  \node[v] (L2b) at (-0.65,1.80) {};
  \draw (L1)--(L2a);
  \draw (L1)--(L2b);

  \node[v] (M2a) at (-0.35,1.80) {};
  \node[v] (M2b) at ( 0.35,1.80) {};
  \draw (M1)--(M2a);
  \draw (M1)--(M2b);

  \node[v] (R2a) at ( 0.65,1.80) {};
  \node[v] (R2b) at ( 1.35,1.80) {};
  \draw (R1)--(R2a);
  \draw (R1)--(R2b);
\end{tikzpicture}
\end{equation}

A list of the (planar version of the) trees $t(n)$ for $n$ ranging from 1 to 102, that perhaps can illustrate what is going on much better than any other abstract explanation, can be found at the end of the paper \cite{CC25}. A very relevant feature for each tree is its \textit{date of birth}, i.e. the first integer associated with it.

\medskip

\noindent
The main result of this work, in simple terms, is the computation of the asymptotic density of any given rooted tree in the natural sequence. We show that such a density depends up to a constant factor (which depends on the tree) only on the \textit{density
signature} of the tree, i.e. a couple of integers $(m,k)$ that is determined by climbing up to the first level (height $1$) the branching from the root. 
We identify among the 
subtrees starting at level~1, the one 
(the ``oldest subtree'') 
having the earliest \textit{date of birth}.
We call $m$ that integer.  We then count how many leaves 
of the original tree contain precisely this specific subtree starting at level~1 and nothing more;
this gives our $k$. For instance, the couples $(m,k)$ for the $6$ previous examples are: 

\[
\begin{array}{c|c}
t(n) & (m,k) \\ \hline
t(16)        & (4,1) \\
t(300)       & (1,1) \\
t(4\,800)      & (1,1) \\
t(307\,200)    & (1,1) \\
t(18\,662\,400)  & (2,1) \\
t(192\,000\,000) & (1,1) \\
t(729\,000\,000) & (6,3)
\end{array}
\]

With these definitions at hand, we can now claim that
the density of a given tree $T$ of signature $(m,k)$
turns out to be
\begin{equation}
m\,\frac{x^{1/m}}{\log x}\,
\frac{(\log\log x)^{k-1}}{(k-1)!}\,
c_T,
\qquad x\to+\infty,
\end{equation}
for a suitable constant $c_T$ that will be carefully identified as the evaluation at $1/m$ of a certain tree zeta function. 
See Theorem \ref{main1} in the main text for the precise statement.

\medskip
The \textit{tree zeta functions}, that will be introduced below (see formula (\ref{zetadef})), are a natural family of functions of one complex variable associated to rooted trees, including as a special case the prime zeta function. 
Along the way, we will also establish some analytic properties of the tree zeta functions. Remarkably, we show that the tree zeta function of a tree as above has an essential singularity at $1/m$ and we describe the corresponding asymptotic behavior in Theorem \ref{main2}.

\medskip
Let us close this introduction by mentioning few special cases of Theorem \ref{main1}.
Of course, the prime number theorem and the subsequent generalization by Landau correspond to the cases $m=1=k$ and $m=1$, $k >1$, respectively. In both these cases, the constant factor is 1 and we recover the aforementioned classical results.
(See example \ref{examples}, (i) for details.)
For a bridge between Naslund work and the present one, 
see example \ref{examples}, (ii).

\medskip

Although for the time being our main result looks quite satisfactory, we feel that it should be possible to improve it to get an estimate that is uniform in $m$ and $k$. 
We plan to come back to this point in the near future.

\medskip
Concerning the notation, in this paper $p,q$ (as well as $p_1, p_2, \ldots, q_1, q_2, \ldots$) always denote prime numbers,
so for instance $\sum_p$, $\sum_{p,q}$, $\prod_p$ should be understood accordingly. Also, $\mathbb{I}_A$ denotes the characteristic function of the set $A$.

	\section{The main Result}
    All the trees in this paper will be rooted and non-planar.
	For two trees $T_1$ and $T_2$, define their product $T_1 \circ T_2 (= T_2 \circ T_1)$ to be the tree obtained as the result of gluing $T_1$ and $T_2$ together at their roots. By $r$ we mean the tree with only one vertex (i.e., the root) and no edges; this is the unit for the product $\circ$, i.e. $T \circ r = r \circ T = T$ for all trees $T$.
    Given a tree $T$, we may consider a new tree $e^T$, where the root of $T$ is attached to the leaf of a tree with a single edge $t(2)$ (the ``prime tree'') and the root of $e^T$ coincides with the root of $t(2)$.

    Given a tree $T$, define the value $\mathfrak{M}(T)$ as the smallest natural number associated with it, 
    namely
	$$\mathfrak{M}(T) = \min\left\{m \in \mathbb{N}: t(m) = T\right\}\!.$$
Therefore, $\mathfrak M$ is a function from the set of all trees ${\mathcal T}$ to $\mathbb N$, with the following properties: 
\begin{itemize}
\item It is injective, and thus it induces a total order on $\mathcal T$: we say that $T < T'$ if ${\mathfrak M}(T) < {\mathfrak M}(T')$. Hence we may define 
$t_k$ to be the $k$-th tree ($t_1 = r$) w.r.t. this order.
\item The image of $\mathfrak M$ is given by
$$
\{1,2,4,6,12,16,30,36,48,60,64,90,144,\ldots\}$$
We observe that these numbers form a strictly increasing sequence $(a_k)_{k=1}^\infty$ that has been defined in \cite[Equation (11)]{CC25}.
\item Of course, $t_k = t(a_k)$.
\end{itemize}

    So, for instance, 
$t(1) = r = t_1$, $t(2)=e^r = t_2 =t(3) \neq t_3$.

	We also set $\mathfrak{M}(\emptyset) = +\infty$,  and $e^\emptyset = r$. 
    
    Finally, for each tree $T$ we define the {\it tree zeta function} $\zeta_T$ by 
	\begin{equation} \label{zetadef}
    \zeta_T(s) = \sum\limits_{\substack{n=1 \\ t(n) = T}}^{+\infty}n^{-s},\ \Re s > 1
    \end{equation}
    (so that $\zeta_{t(1)} = 1$ and $\zeta_{t(2)} = P$, i.e. the zeta function of the prime tree is the prime zeta function).
        In passing, we observe that one can recover the Riemann zeta function by summing the tree zeta-functions over all trees, namely
    $$\zeta(s) = \sum_T \zeta_T(s) \ . $$ 

	For $x \in {\mathbb R}$, $x \geq 1$, we also set 
	$$\pi_T(x) = \#\left\{n\leqslant x: t(n) = T\right\}\!.$$
    For instance, $\pi_{t(2)} $  coincides with the prime-counting function $\pi$.
    Providing suitable asymptotic estimates for $x \to +\infty$ for such tree-counting functions $\pi_T$ will be the main focus of this paper. 
    Notice that $\sum_{T \in {\mathcal T}(x)} \pi_T(x) = \lfloor x \rfloor$ for all $x \geq 1$, where ${\mathcal T}(x) = \{T \in {\mathcal T} \ | \ {\mathfrak M}(T) \leq x\}$ is the subset of trees that appear in the window up to $x$.

We start showing that any nontrivial tree $\mathfrak{T}$ can be uniquely written as the product of two trees $T \circ T'$ in a way that is suitable for our later purposes. The easy proof is omitted.
    \begin{lem}
    For any given tree $\mathfrak{T} \in \mathcal{T}$
     different from $r$
    there exist unique trees $T_0, T'$ and integers $k\geqslant 1, s\geqslant 0$ such that $$\mathfrak{T} = T\circ T',$$
    where
    $$T = \underbrace{ {e^{T_0}\circ e^{T_0}\circ \cdots \circ e^{T_0}}}_{k\ \textup{times}},\quad T' = \begin{cases} t(1), & s=0, \\ e^{T_1}\circ e^{T_2}\circ \cdots \circ e^{T_s}, & s >0, \end{cases} $$
    and, in the latter case, for each $1\leqslant j\leqslant s$, we have
		$\mathfrak{M}(T_0)<\mathfrak{M}(T_j) 
    < + \infty$.
\end{lem}

\medskip
We can also refine the information about the region of convergence of the series defining the tree zeta-functions in the following way.
\begin{lem}
Let ${\mathfrak{T}} = e^{T_0} \circ \cdots \circ e^{T_0} \circ e^{T_1} \circ \cdots \circ e^{T_s}$
($\ell$ factors in total), with $m := {\mathfrak M}(T_0) < {\mathfrak M}(T_j) < + \infty$ for all $1 \leq j \leq s$.
Then the series 
$$\zeta_{\mathfrak T}(s) = \sum\limits_{\substack{n=1 \\ t(n) = {\mathfrak T}}}^{+\infty}n^{-s}$$
is absolutely convergent for $\Re s > \frac 1 m$.
\end{lem}
\begin{proof}
Indeed, for $s = \sigma + it$ with $\sigma>\frac{1}{m}$, using the fact that if $t(n) = {\mathfrak T}$ then $n = p_1^{\alpha_1} \ldots p_\ell^{\alpha_\ell}$ with $\alpha_h \geq m$ for all $1 \leq h \leq \ell$, 
we have
\begin{multline*}|\zeta_\mathfrak{T}(s)|\leqslant \sum\limits_{\substack{n = 1 \\ t(n) = \mathfrak{T}}}^{+\infty}n^{-\sigma}\leqslant \sum\limits_{\substack{n = 1 \\ p|n \Rightarrow p^m|n}}^{+\infty}n^{-\sigma}\\
= \prod_{p}\left(1+p^{-m\sigma}+p^{-(m+1)\sigma}+\cdots\right) = \prod_{p}\left(1+\dfrac{1}{p^{m\sigma}-p^{(m-1)\sigma}}\right)<+\infty. 
\end{multline*}
This concludes the proof.
\end{proof}
In the appendix, we prove that
such a series has an essential singularity at $s = 1/m$.

    \medskip
	After the previous lemmata,
    we are now ready to prove the following result.
	\begin{theorem}\label{main1}
Let $\mathfrak{T}$ be a tree with at least one edge,
written uniquely in the form
$\mathfrak{T} = T\circ T'$, where 
		$$ T = \underbrace{ {e^{T_0}\circ e^{T_0}\circ \cdots \circ e^{T_0}}}_{k\ \textup{times}} $$
        for some tree $T_0$ and $k \in {\mathbb N}$,
        and $T'$ is either 
$t(1)$ or
        $$e^{T_1}\circ e^{T_2}\circ \cdots \circ e^{T_s} $$
where, for each $1\leqslant j\leqslant s$, we have
		$m := \mathfrak{M}(T_0)<\mathfrak{M}(T_j) 
        < +\infty.$ 
        
Then,
		$$\pi_{\mathfrak{T}}(x) = \dfrac{mx^{\frac{1}{m}}}{\log x}\dfrac{(\log\log x)^{k-1}}{(k-1)!}\times\zeta_{T'}\left(\dfrac{1}{m}\right) + R_{\mathfrak{T}}(x),$$
		with
		$$
		R_{\mathfrak{T}}(x)\ll_{m,  k}\begin{cases}
			\dfrac{x^{\frac{1}{m}}(\log\log x)^{k-2}}{\log x},\ \ \text{if}\ \ k\geqslant 2,\\
			\dfrac{x^{\frac{1}{m}}\log\log x}{(\log x)^2},\ \ \text{if}\ \ k = 1.
		\end{cases}
		$$
\end{theorem}
\begin{rem}
We want to emphasize the role of the tree $T_0$ in the theorem, because it fully captures the functional description of $\pi_{\mathfrak T}$ through its multiplicity $k$ and order $m$.
In particular, the tree $T'$ is only responsible for the overall multiplicative factor $\zeta_{T'}(1/m)$.
\end{rem}
\begin{proof}
For any $n$ with $t(n) = \mathfrak{T}$, we have the following unique representation
$n = ab,$ where $t(a) = T, t(b) = T',$ and $\gcd(a, b) = 1$. Set
$$c = \prod\limits_{\substack{p|a \\ \nu_p(a) = m }}p,\ \ 
a_1 = c^m, \ \ \text{and}\ \ a_2 = \prod\limits_{\substack{p|a \\ \nu_p(a) > m}}p^{\nu_p(a)}$$
(if there is no $p$ such that $p | a$ and $\nu_p(a)=m$ we set $c=1$).
Then $a_1$ is coprime with $a_2$ and $a=a_1a_2$. Put $d = a_2b$, then $n$ has the unique representation in the form $n=c^md$, and $\gcd(c,d)=1$. Moreover, we have $p|d \Rightarrow \nu_p(d)>m,$ and
$$t(d) = \underbrace{e^{T_0}\circ e^{T_0}\circ \cdots \circ e^{T_0}}_{{k-\omega(c)}}\circ\,T' = : \mathfrak{T}_{\omega(c)}.$$
Let $P_\ell$ denotes the contribution of those $n\leqslant x$, for which $t(n) = \mathfrak{T}$ and $\omega(c) = \ell$. Therefore,
$$\pi_{\mathfrak{T}}(x) = \sum_{\ell = 0}^k P_\ell,\ \text{where}\ \  P_\ell = \sum\!\!\!\!\!\!\!\!\!\!\!\!\!\!\!\sum\limits_{\substack{c^md\leqslant x,\  \gcd(c,d) = 1 \\ \omega(c) = \ell,\ t(d) = \mathfrak{T}_\ell \\ p|d \Rightarrow \nu_p(d)>m}}\!\!\!\!\!\!\!\!\!\!\!\!\mu^2(c).$$
Using the asymptotic formula (see \cite{IS82})
\begin{equation}\label{full}
	\#\left\{d\leqslant Y: p|d \Rightarrow \nu_p(d)>m\right\} = C_m Y^{\frac{1}{m+1}}(1+o(1)),\ \ (C_m>0,\ \ Y\to+\infty),
	\end{equation}
we conclude that 
$$P_0\leqslant \sum\limits_{\substack{d\leqslant x \\ p|d \Rightarrow \nu_p(d)>m}} 1\ll_m x^{\frac{1}{m+1}}.$$
Now let $\ell\geqslant 1$, then
\begin{equation}\label{P11}
	P_\ell = \bigl(\!\!\!\! \sum\limits_{\substack{d\leqslant \mathcal{H} \\ t(d) = \mathfrak{T}_\ell \\ p|d\Rightarrow \nu_p(d)>m}} + \sum\limits_{\substack{ \mathcal{H}<d\leqslant x \\ t(d) = \mathfrak{T}_\ell \\ p|d\Rightarrow \nu_p(d)>m}}\!\!\bigr) \sum\limits_{\substack{c\leqslant y \\\gcd(c, d) = 1\\ \omega(c) = \ell}}\mu^2(c) = P_\ell^{(1)}+P_\ell^{(2)},\ \ \text{say,}
	\end{equation}
    where for shortness we set $y = (x/d)^{{1}/{m}},$ and $\mathcal{H} = x^{o(1)}$ to be defined later.
First, we estimate the sum $P_\ell^{(2)}$. We have $P_\ell^{(2)} = P_\ell^{(3)}+P_\ell^{(4)}$, where $P_\ell^{(3)}$ denotes the contribution of $\mathcal{H}<d\leqslant \sqrt{x}$, and $P_\ell^{(4)}$ denotes the contribution of the remaining $d$.
Consider the series 
\begin{equation}\label{S_m}
	S_m(X) = \sum\limits_{\substack{d>X \\ p|d \Rightarrow \nu_p(d)>m}}\dfrac{1}{d^{1/m}}.
	\end{equation}
Using partial summation and \eqref{full},
we obtain
\begin{equation*}
	S_m(X) = \dfrac{1}{m}\int_X^{+\infty}\dfrac{\#\left\{X<d\leqslant t: p|d\Rightarrow \nu_p(d)>m\right\}}{t^{1+\frac{1}{m}}} dt\ll_m\int_X^{+\infty}\dfrac{dt}{t^{1+\frac{1}{m(m+1)}}}\ll_m\dfrac{1}{X^{\frac{1}{m(m+1)}}}.
\end{equation*}
Hence, using the Hardy~--~Ramanujan inequality (see \cite{HR17}), 
$$\#\left\{n\leqslant x: \omega(n) = v\right\}\ll \dfrac{x(\log \log x+c_0)^{v-1}}{(v-1)!\log x},$$
where $x\geqslant 2$, $v\geqslant 1$, $c_0>0$, and the implied constant is absolute, we get
\begin{multline*}
	P_\ell^{(3)}\leqslant \sum\limits_{\substack{\mathcal{H}<d\leqslant \sqrt{x} \\ p|d\Rightarrow \nu_p(d)>m}}\sum\limits_{\substack{\omega(c) = \ell\\ c\leqslant (\frac{x}{d})^{1/m}}} 1\ll_m \sum\limits_{\substack{\mathcal{H}<d\leqslant \sqrt{x} \\ p|d\Rightarrow \nu_p(d)>m}}\left(\dfrac{x}{d}\right)^{\frac{1}{m}} \dfrac{(\log\log x + c_0 )^{\ell-1}}{(\ell-1)!\log\frac{x}{d}}\\
	\ll \dfrac{x^{\frac{1}{m}}(\log\log x+c_0)^{\ell-1}}{\log x}\sum\limits_{\substack{d>\mathcal{H} \\ p|d \Rightarrow \nu_p(d)>m}}\dfrac{1}{d^{1/m}}\ll_{m, k} \dfrac{x^{\frac{1}{m}}(\log\log x)^{\ell-1}}{\mathcal{H}^{\frac{1}{m(m+1)}}\log x}.
\end{multline*}
Similarly, we find that
$$P_\ell^{(4)}\ll \sum\limits_{\substack{d>\sqrt{x} \\ p|d\Rightarrow \nu_p(d)>m}}\left(\dfrac{x}{d}\right)^{\frac{1}{m}}\ll x^{\frac{1}{m}-\frac{1}{2m(m+1)}}.$$
Thus, for $1\leqslant \ell\leqslant k$, we have
\begin{equation}\label{P22}
	P_\ell = P_\ell^{(1)}+O_{k,m}\left(\dfrac{x^{\frac{1}{m}}(\log\log x)^{\ell-1}}{\mathcal{H}^{\frac{1}{m(m+1)}}\log x} + x^{\frac{1}{m}-\frac{1}{2m(m+1)}}\right)\!.
	\end{equation}
Consider now the sum $P_\ell^{(1)}$. We have 
\begin{equation}\label{P33}
	P_\ell^{(1)} = \sum\limits_{\substack{d\leqslant \mathcal{H} \\ t(d) = \mathfrak{T}_\ell \\ p|d\Rightarrow \nu_p(d)>m}}\sum\limits_{\substack{c \leqslant y \\ \gcd(c, d) = 1 \\ \omega(c) = \ell}}\mu^2(c) = P_\ell^{(5)}- \mathcal{M}_\ell,
	\end{equation}
where
$$ P_\ell^{(5)} = \sum\limits_{\substack{d\leqslant \mathcal{H} \\ t(d) = \mathfrak{T}_\ell \\ p|d\Rightarrow \nu_p(d)>m}}\sum\limits_{\substack{c \leqslant y \\ \omega(c) = \ell}}\mu^2(c),
$$
and
$$
\mathcal{M}_\ell = \sum\limits_{\substack{d\leqslant \mathcal{H} \\ t(d) = \mathfrak{T}_\ell \\ p|d\Rightarrow \nu_p(d)>m}}M_\ell(x; d),$$
$$M_\ell(x;d) = \sum\limits_{\substack{c \leqslant y \\ \gcd(c, d) > 1 \\ \omega(c) = \ell}}\mu^2(c) = \sum\limits_{\substack{\Delta | d \\ \Delta > 1}} \sum\limits_{\substack{c \leqslant y \\ \gcd(c, d) = \Delta \\ \omega(c) = \ell}}\mu^2(c).$$
Hence, 
$$
M_\ell(x;d) = \sum\limits_{\substack{\Delta | d \\ \Delta > 1}}\sum\limits_{\substack{c' \leqslant \frac{y}{\Delta} \\ \gcd(c', \frac{d}{\Delta}) = 1 \\ \omega(c'\Delta) = \ell}}\mu^2(c'\Delta) = \sum\limits_{\substack{\Delta | d \\ \Delta > 1 \\ \omega(\Delta)\leqslant \ell }}\mu^2(\Delta)\sum\limits_{\substack{c' \leqslant \frac{y}{\Delta} \\ \gcd(c', d) = 1 \\ \omega(c') = \ell-\omega(\Delta)}}\mu^2(c').
$$
Therefore, we have
\begin{multline}
	M_\ell(x;d) \leqslant \sum_{v=0}^{\ell-1}\sum\limits_{\substack{ \Delta | d\\ \Delta > 1 \\ \omega(\Delta) = \ell - v}}\mu^2(\Delta)\sum\limits_{\substack{c'\leqslant \frac{y}{\Delta} \\ \omega(c') = v}}1\\
	\ll \sum\limits_{\substack{\Delta|d \\ \Delta >1}}\mu^2(\Delta)+ \sum_{v=1}^{\ell-1}\sum\limits_{\substack{ \Delta | d\\ \Delta > 1 \\ \omega(\Delta) = \ell - v}}\mu^2(\Delta)\dfrac{y}{\Delta \log\frac{y}{\Delta}}\dfrac{(\log\log \frac{y}{\Delta}+c_0)^{v-1}}{(v-1)!}.
\end{multline}
If $\ell = 1$, then the second sum is empty. Further, since $y/\Delta \geqslant x^{1/m-o(1)}$
and
$$\sum_{\Delta|d}\dfrac{\mu^2(\Delta)}{\Delta} = \dfrac{d}{\varphi(d)},$$
we conclude that
	$$M_\ell(x;d)\ll_{m, k} \tau(d) + \mathbb{I}_{\ell\geqslant 2}\dfrac{y(\log\log x)^{\ell-2}}{\log x}\dfrac{d}{\varphi(d)}.
$$
Hence, using the estimate
$$\sum_{d\leqslant Y}\tau(d)\ll Y\log Y,$$
we see that
\begin{equation*}
	\mathcal{M}_\ell \ll_{m, k} \mathcal{H}\log\mathcal{H} + \mathbb{I}_{\ell\geqslant 2} 	\dfrac{x^{\frac{1}{m}}(\log\log x)^{\ell-2}}{\log x}\sum\limits_{\substack{d\leqslant \mathcal{H} \\ t(d) = \mathfrak{T}_\ell \\ p|d\Rightarrow \nu_p(d)>m}}\dfrac{d^{1-\frac{1}{m}}}{\varphi(d)}.
	\end{equation*}
The last sum does not exceed
\begin{equation*}
	D_m(\mathcal{H}) = \sum\limits_{\substack{d\leqslant \mathcal{H} \\ p|d\Rightarrow \nu_p(d)>m}}\dfrac{d^{1-\frac{1}{m}}}{\varphi(d)} = \sum\limits_{\substack{d\leqslant \mathcal{H}  \\ p|d\Rightarrow \nu_p(d)>m}}{d^{-\frac{1}{m}}}\sum_{\Delta| d}\dfrac{\mu^2(\Delta)}{\Delta} = \sum_{\Delta\leqslant \mathcal{H}}\dfrac{\mu^2(\Delta)}{\varphi(\Delta)}\sum\limits_{\substack{d\leqslant \mathcal{H} \\d\equiv 0\!\!\pmod{\Delta}\\ p|d\Rightarrow \nu_p(d)>m} }d^{-\frac{1}{m}}.
	\end{equation*}
Each $d$ in the inner sum has a unique representation in the form
$$d = \Delta\Delta_1\delta,$$
where $(\delta,\Delta) = 1$ and $\Delta_1|\Delta^\infty$ (meaning that there is some $N\geq 1$ such that $\Delta_1$ divides $\Delta^N$). Therefore,
$$D_m(\mathcal{H})\leqslant \sum_{\Delta\leqslant \mathcal{H}} \dfrac{\mu^2(\Delta)}{\Delta^{\frac{1}{m}}\varphi(\Delta)}\sum_{\Delta_1|\Delta^\infty}\Delta_1^{-\frac{1}{m}}\times\sum\limits_{\substack{\delta = 1 \\ p|\delta \Rightarrow \nu_p(\delta)>m}}^{+\infty}\delta^{-\frac{1}{m}}.$$
Since
$$\sum_{\Delta_1|\Delta^\infty}\Delta_1^{-\frac{1}{m}} = \prod_{p|\Delta}\left(1+p^{-\frac{1}{m}}+p^{-\frac{2}{m}}+\cdots\right) = \prod_{p|\Delta}\left(1+\dfrac{1}{p^{\frac{1}{m}}-1}\right)\!,$$
$$p^{\frac{1}{m}}-1\geqslant 2^{\frac{1}{m}}-1\geqslant {\log 2}/{m},$$
it follows that
$$ D_m(\mathcal{H})\leqslant \sum_{\Delta\leqslant \mathcal{H}} \dfrac{\mu^2(\Delta)C(m)^{\omega(\Delta)}}{\Delta^{\frac{1}{m}}\varphi(\Delta)} S_m\left(1/2 \right)\ll_m \sum_{\Delta = 1}^{+\infty} \Delta^{-1-\frac{1}{m}+o(1)}\ll_m 1,$$
where $S_m$ is defined in \eqref{S_m}, and $C(m) = m/\log 2 + 1.$ Thus,
\begin{equation}\label{P44}
	\mathcal{M}_\ell \ll_{m, k} \mathcal{H}\log\mathcal{H} + \mathbb{I}_{\ell\geqslant 2} 	\dfrac{x^{\frac{1}{m}}(\log\log x)^{\ell-2}}{\log x}.
	\end{equation}
Consider now the sum $P_\ell^{(5)}$.  Using the asymptotic formula (see \cite{Lan})
$$\sum\limits_{\substack{a\leqslant x \\ \omega(a) = k}}\mu^2(a) = \dfrac{x }{\log x}\dfrac{(\log\log x)^{k-1}}{(k-1)!} + R_k(x),$$
where
$$R_k(x)\ll \begin{cases}
	\dfrac{x}{(\log x)^2},\ \ \text{if\ \ } k=1,\\
	\\
	\dfrac{x(\log\log x)^{k-2}}{\log x},\ \ \text{if\ \ }k\geqslant 2,
\end{cases}
$$
we obtain
\begin{equation}\label{Main}
P_\ell^{(5)} = \sum\limits_{\substack{d\leqslant \mathcal{H} \\ t(d) = \mathfrak{T}_\ell \\ p|d\Rightarrow \nu_p(d)>m}}\left(\left(\dfrac{x}{d}\right)^{{1}/{m}}\dfrac{1}{\log(({x}/{d})^{\frac{1}{m}}) }\dfrac{(\log\log((x/d)^{\frac{1}{m}}))^{\ell-1}}{(\ell-1)!}+R_\ell\left( \left(\dfrac{x}{d}\right)^{{1}/{m}}  \right) \right)\!.
\end{equation}
Let $\mathcal{R}_\ell$ denote the contribution of $R_\ell$ to the sum over $d$. Then for $\ell = 1$ we have
$$\mathcal{R}_1 \ll \sum\limits_{\substack{d\leqslant \mathcal{H} \\ t(d) = \mathfrak{T}_\ell \\ p|d\Rightarrow \nu_p(d)>m}}\left(\dfrac{x}{d}\right)^{{1}/{m}}\dfrac{1}{\log^2(({x}/{d})^{\frac{1}{m}})}\ll_m\dfrac{x^{\frac{1}{m}}}{(\log x)^2}.$$
For $\ell\geqslant 2$ we have
$$\mathcal{R}_\ell \ll_m \sum\limits_{\substack{d\leqslant \mathcal{H} \\ t(d) = \mathfrak{T}_\ell \\ p|d\Rightarrow \nu_p(d)>m}} \left(\dfrac{x}{d}\right)^{{1}/{m}}\dfrac{(\log\log((x/d)^{\frac{1}{m}}))^{\ell-2}}{\log(({x}/{d})^{\frac{1}{m}}) }\ll_m \dfrac{x^{\frac{1}{m}}(\log\log x)^{\ell-2}}{\log x}.$$
Thus,
\begin{equation}\label{P55}
	\mathcal{R}_\ell \ll_m \mathbb{I}_{\ell = 1}\dfrac{x^{\frac{1}{m}}}{(\log x)^2}+\mathbb{I}_{\ell\geqslant 2}\dfrac{x^{\frac{1}{m}}(\log\log x)^{\ell-2}}{\log x}.
	\end{equation}
Denote by $\mathcal{P}_\ell$ the contribution of the first term in \eqref{Main} to the sum over $d$. Since $\mathcal{H} = x^{o(1)}$ and
$$\log\left(\left( \dfrac{x}{d}\right)^{{1}/{m}}\right) = \dfrac{\log x}{m}\left(1-\dfrac{\log d}{\log x}\right) = \dfrac{\log x}{m}\left(1+O\left(\dfrac{\log \mathcal{H}}{\log x}\right)\right)\!,$$
we obtain
\begin{multline*}\log\log\left(\left( \dfrac{x}{d}\right)^{{1}/{m}}\right) = \log\log x-\log m+O\left(\dfrac{\log \mathcal{H}}{\log x}\right)
	= \log\log x\left(1+O_m\left(\dfrac{1}{\log\log x} \right)\right)\!.
	\end{multline*}
Therefore, for $\ell\geqslant 1$ we have
\begin{multline}\label{P66}
	\mathcal{P}_\ell = \sum\limits_{\substack{d\leqslant \mathcal{H} \\ t(d) = \mathfrak{T}_\ell \\ p|d\Rightarrow \nu_p(d)>m}}\left(\dfrac{x}{d}\right)^{{1}/{m}}\dfrac{m}{\log x} \dfrac{(\log\log x)^{\ell-1}}{(\ell-1)!}\\
	\times\left(1+O_m\left(\dfrac{\log\mathcal{H}}{\log x} \right) \right)\left(1+O_m\left(\dfrac{1}{\log\log x}\right) \right)^{\mathbb{I}_{\ell\geqslant 2}}\\
	= \dfrac{m x^{\frac{1}{m}}}{\log x}\dfrac{(\log\log x)^{\ell-1}}{(\ell-1)!}\sum\limits_{\substack{d = 1 \\ t(d) = \mathfrak{T}_\ell \\ p|d\Rightarrow \nu_p(d)>m}}^{+\infty} d^{-\frac{1}{m}} + \mathcal{R}_\ell',
	\end{multline}
where
\begin{equation}\label{P77}
	\mathcal{R}_\ell'\ll_{m, k} \dfrac{x^{\frac{1}{m}}(\log\log x)^{\ell-1}}{\log x}\left({\mathcal{H}^{-\frac{1}{m(m+1)}}}+\dfrac{\mathbb{I}_{\ell\geqslant 2}}{\log\log x} + \dfrac{\log\mathcal{H}}{\log x}\right)\!.
	\end{equation}
Thus, it follows from \eqref{P11}, \eqref{P22}, \eqref{P33}, \eqref{P44}~--~\eqref{P77} that for $1\leqslant \ell \leqslant k$,
$$P_\ell = P_\ell^{MT} + P_\ell^{Err},$$
where 
$$P_\ell^{MT} = m C_m(\ell)\dfrac{ x^{\frac{1}{m}}}{\log x}\dfrac{(\log\log x)^{\ell-1}}{(\ell-1)!},\ \ \ C_m(\ell)=\!\!\!\!\!\!\!\!\!\! \sum\limits_{\substack{d = 1 \\ t(d) = \mathfrak{T}_\ell \\ p|d\Rightarrow \nu_p(d)>m}}^{+\infty} d^{-\frac{1}{m}},$$
and
\begin{multline}P^{Err}_\ell\ll_{m, k} \mathcal{H}\log\mathcal{H} + \mathbb{I}_{\ell = 1}\dfrac{x^{\frac{1}{m}}}{(\log x)^2} \\
	+ \dfrac{x^{\frac{1}{m}}(\log\log x)^{\ell-1}}{\log x}\left({\mathcal{H}^{-\frac{1}{m(m+1)}}}+ \dfrac{\mathbb{I}_{\ell\geqslant 2}}{\log\log x}+ \dfrac{\log\mathcal{H}}{\log x}\right) + x^{\frac{1}{m}-\frac{1}{2m(m+1)}}.
	\end{multline}
Choosing $\mathcal{H} = (\log x/\log\log x)^{m(m+1)}$, we get
$$P^{Err}_\ell \ll_{m, k} \mathbb{I}_{\ell = 1}\dfrac{x^{\frac{1}{m}}\log\log x}{(\log x)^2} + \mathbb{I}_{\ell\geqslant 2} \dfrac{x^{\frac{1}{m}}(\log\log x)^{\ell-2}}{\log x}.$$

It turns out that the main contribution to $\pi_{\mathfrak{T}}(x)$ comes from the term with $\ell=k$. Define $\pi^{MT}_{\mathfrak{T}}(x) = P^{MT}_k$. Consider two cases. First, let $T' \neq t(1)$, then for the main term, $\pi^{MT}_{\mathfrak{T}}(x)$, we have

$$\pi^{MT}_{\mathfrak{T}}(x) = mC_m(k)\dfrac{x^{\frac{1}{m}}}{\log x}\dfrac{(\log\log x)^{k-1}}{(k-1)!},$$
where
$$C_m(k) = \sum\limits_{\substack{d = 1 \\ t(d) = \mathfrak{T}_0 \\ p|d\Rightarrow \nu_p(d)>m}}^{+\infty}d^{-\frac{1}{m}} = \sum\limits_{\substack{d = 1 \\ t(d) = T'}}^{+\infty}d^{-\frac{1}{m}} = \zeta_{T'}\left(\dfrac{1}{m}\right)\!.$$
On the other hand, if $T' = t(1)$, then one gets again
$$C_m(k) = \sum\limits_{\substack{d = 1 \\ t(d) = t(1)}}^{+\infty}d^{-\frac{1}{m}} = 1 = \zeta_{T'}\left(\dfrac{1}{m}\right)\!.$$
Thus, we have
$$\pi_{\mathfrak{T}}^{MT}(x) = \dfrac{m x^{\frac{1}{m}}}{\log x}\dfrac{(\log\log x)^{k-1}}{(k-1)!}\times\zeta_{T'}\left(\dfrac{1}{m}\right)\!.$$
Define the remainder term, $\pi_{\mathfrak{T}}^{Err}(x),$ to be
$$\pi_{\mathfrak{T}}^{Err}(x) = P^{Err}_k + P_0 + \sum_{\ell = 1}^{k-1} P_\ell.$$ Then we have $$\pi_{\mathfrak{T}}(x) = \pi_{\mathfrak{T}}^{MT}(x)+\pi_{\mathfrak{T}}^{Err}(x)$$
and
\begin{equation}\pi^{Err}_{\mathfrak{T}}(x) \ll \mathbb{I}_{k=1}\dfrac{x^{\frac{1}{m}}}{(\log x)^2} + \mathbb{I}_{k\geqslant 2} \dfrac{x^{\frac{1}{m}}(\log\log x)^{k-2}}{\log x}.
\end{equation}
This concludes the proof.
\end{proof}

Below we present a number of special cases.
\begin{ex}\label{examples}
\begin{itemize}
\item[(i)] Consider $\mathfrak{T} = T \circ T'$, where $$T = e^r \circ e^r \circ \cdots \circ e^r$$ (product of $k$ copies of $e^r$, with $k \geq 1$), $T' = r = e^{\emptyset}$ so that
$m={\mathfrak M}(r) = 1 < \mathfrak{M}(\emptyset) = +\infty$ and $\zeta_{r}(1) = 1$. Then $\mathfrak T$ is nothing but a tree with $k$ edges emanating from the root and the formula for $\pi_\mathfrak{T}(x)$ reproduces either the prime number theorem ($k=1$) or Landau's result about the asymptotic density of numbers that are products of $k$ distinct primes ($k > 1$).

\item[(ii)] Let $\mathfrak{T}=t(12)$, then $T = t(2)$ (so that $k=1$, $T_0 = r$ and $m=1$) and $T' = t(4)$. Thus we get
$$\pi_{\mathfrak T}(x) \sim 
\frac {x}{\log x} \, \zeta_{t(4)}(1)$$ for $x \to \infty$, where
$$\zeta_{t(4)}(1) = \sum_{p,q} \frac 1 {p^q} 
< + \infty\ . $$
This result can be deduced from the work of Naslund in the following way.
Noticing that by \cite{Nas} one has 
$\#\{n \leq x \ | \ n = p q^2, p \neq q\} 
\sim \frac x {\log x} P(2) 
= \frac x {\log x} \sum_p \frac 1 {p^2}$ for $x \to \infty$
and, more generally, for every prime $r$,
$\#\{n \leq x \ | \ n = p q^r, p \neq q\} \sim \frac x {\log x} P(r) 
= \frac x {\log x} \sum_p \frac 1 {p^r} $ for $x \to \infty$,
one gets indeed
$$\pi_{\mathfrak T}(x) = \sum_r \#\{n \leq x \ | \ n = p q^r, p \neq q\} \sim \frac x {\log x} \sum_r\sum_p \frac 1 {p^r} \ , \ x \to +\infty. $$

\item As another example consider $\mathfrak{T} = t(331\,776)$, $331\,776 = 2^{2^2\cdot 3} \cdot 3^{2^2}$.
Here, $T = e^{t(4)} = t(16)$, $T' = e^{t(12)} = t(2^{12})$ (where $T_0 = t(4)$ and $T_1 = t(12) $) and $(m,k)=(4,1)$. Moreover,
$$\zeta_{t(2^{12})}(s) = \sum\limits_{\substack{p_1, \ldots, p_4\\ p_2\neq p_4}} \frac{1}{p_1^{p_2^{p_3} p_4 \, s}}.$$
Then $$\pi_{\mathfrak{T}}(x) \sim \dfrac{4x^{\frac{1}{4}}}{\log x}\zeta_{t(2^{12})}\left(\dfrac{1}{4}\right)\!, \ x \to +\infty.$$
\item If we take $\mathfrak{T} = t(207\,360\,000)$, $207\,360\,000 = 2^{2^2\cdot 3} \cdot 3^{2^2} \cdot 5^{2^2}$, then as above we will have $T_0 = t(4), T_1 = t(12)$, $T' = t(2^{12})$, but $(m, k) = (4, 2)$. Therefore,
$$\pi_\mathfrak{T}(x) \sim \dfrac{4x^{\frac{1}{4}}\log\log x}{\log x}\zeta_{t(2^{12})}\left(\dfrac{1}{4}\right)\!, \ x \to +\infty. $$
\end{itemize}
\end{ex}

\section{Conclusions and outlook}\label{outlook}

In this paper we computed the asymptotic density of integers with a given tree structure, namely
$$\pi_{\mathfrak{T}}(x) \sim \dfrac{mx^{\frac{1}{m}}}{\log x}\dfrac{(\log\log x)^{k-1}}{(k-1)!}\times\zeta_{T'}\left(\dfrac{1}{m}\right),\quad x\to +\infty.$$
This formula includes as special cases classical results as the prime number theorem and later work by Landau.

It would be interesting to determine if the asymptotic densities of different trees $\mathfrak{T}=T \circ T'$ and $\mathfrak{T}'= T \circ T''$ and $T''$ with the same values of the pair $(m,k)$ are distinguished by the corresponding values of the zeta functions $\zeta_{T'}(1/m)$ and $\zeta_{T''}(1/m)$. Of course, this would mean that different trees have different densities.


\medskip
We actually plan to get a "uniform" version of this result (meaning that the estimates are uniform in $m$ and $k$).

Using the Selberg-Delange method (see e.g. \cite[Chapter II.5]{Ten}), and more specifically applying \cite[Theorem 6.1]{Ten} to 
$$\sum\limits_{\substack{n\le x, \\ \gcd(n, d) = 1}} z^{\o(n)}\mu^2(n) \ , $$
one can show by Theorem 6.3 therein the asympotic formula
\begin{equation}\label{asymptform}
\sum\limits_{\substack{n \leqslant x:\\ \omega (n) = \ell \\ \gcd(n, d) = 1}} \mu^2(n) = \dfrac{x}{\log x}\dfrac{(\log\log x)^{\ell-1}}{(\ell-1)!}\left\{\lambda_d\left(\dfrac{\ell-1}{\log\log x}\right)+O_\varepsilon\left(\dfrac{\ell c(d)} {(\log\log x)^2}\right)\right\},
\end{equation}
where $1\leqslant \ell \leqslant A\log\log x,$ $0<A\leqslant 2-\varepsilon$,
$c(d)\ll_\varepsilon \exp((\omega(d)^\varepsilon))$ and
$$\lambda_d(z) = \dfrac{1}{\Gamma(z+1)}\prod_p\left\{\left(1+\dfrac{z}{p}\right)\left(1-\dfrac{1}{p}\right)^z\right\}\prod_{p|d}\left(1+\dfrac{z}{p}\right)^{-1} \ . $$
One should then be able to apply it to get an asymptotic formula for $\pi_T(x)$, displaying the desired uniform behaviour.
Notice that for $z \to 0$ we have that $\lambda_d \to 1$, meaning that formula \eqref{asymptform} reproduces Landau's theorem.

\begin{rem} Among some other related issues that we plan to discuss in the near future we mention:
\begin{itemize}
\item 
It is possible to obtain an asymptotic expansion for $\pi_T(x)$ in the form of an asymptotic series 
(cf. e.g. \cite{De71,CrEr21}).
\item 
The uniform estimate mentioned above 
could hopefully help to identify the highest value of 
$\frac {\pi_T(x)}{x}$ for a given $x$ 
(it might still happen that the same value could be shared by different trees).
\item 
It would be interesting to recover Zipf's law 
for the rank 
of trees appearing up to $x$ (see \cite{CGOV})
on theoretical grounds.
\item 
A further study of the analytic properties of the tree zeta functions seems an interesting topic on its own. (E.g., their analytic continuation, behaviour at singular points, ....).
\end{itemize}
\end{rem}

\medskip
\noindent{\bf Acknowledgments}
R.C. is partially supported by Sapienza Universit\`a di Roma. P.C. is partially supported by Alma Mater Studiorum
University of Bologna, by PRIN project 2022 N.2022B5LF52
and by the Italian Extended Partnership PE01--FAIR (Future Artificial Intelligence Research, proposal code PE00000013) under the National Recovery and Resilience Plan. The study of V.I. was implemented in the framework of the Basic Research Program at HSE University in 2025.

\section{Appendix}

Let ${\mathcal P} \subset {\mathbb N}$ be a subset of the natural numbers, 
$k \in {\mathbb N}$ and 
${\mathcal F}: {\mathcal P}^k \to {\mathbb C}$ be a $\mathbb C$-valued function. We are interested in evaluating the sum of ${\mathcal F}(p_1,\ldots,p_k)$ over all $k$-tuples $(p_1,\ldots,p_k)$ of {\it distinct} elements from $\mathcal P$. We set $[k]:=\{1,2,\ldots,k\}$.

We have the following partition-based inclusion-exclusion formula.
\begin{theorem} 
Suppose that the series $\sum_{p_1, \ldots, p_k \in {\mathcal P}} {\mathcal F}(p_1,\ldots,p_k)$ is absolutely convergent. Then one has
\begin{equation} 
\label{sumformula}
\sum\limits_{\substack{p_1,\ldots,p_k \in {\mathcal P}\\ 
i \neq j \Rightarrow p_i \neq p_j }}
{\mathcal F}(p_1,\ldots,p_k)
= \sum_{\ell = 1}^k (-1)^{k - \ell}
\sum\limits_{\substack{A_1 \sqcup \cdots \sqcup A_\ell = [k]\\ 
i < j \Rightarrow \min A_i < \min A_j}}c_{\mathcal{A}_1, \ldots, \mathcal{A}_\ell}
\sum\limits_{\substack{q_1,\ldots,q_\ell \in {\mathcal P}
}}
{\mathcal F}(p_1,\ldots,p_k),
\end{equation}
where in the last sum we put $p_j = q_v\,$ for each $j \in A_v $ with $1 \leq v \leq \ell,$ and
$$c_{\mathcal{A}_1, \ldots, \mathcal{A}_\ell} = \left(\#\mathcal{A}_1-1\right)!\cdots \left(\#\mathcal{A}_\ell-1\right)!.$$
\end{theorem}
\begin{proof}
Let $$F = \sum\limits_{\substack{p_1,\ldots,p_k \in {\mathcal P}\\ 
i \neq j \Rightarrow p_i \neq p_j }}
{\mathcal F}(p_1,\ldots,p_k).$$ Then we have
\begin{multline*}
F = \sum_{p_1}\sum_{p_2 \not{\in} \left\{p_1\right\}}\cdots\sum_{p_k \not{\in}\left\{p_1, p_2, \ldots, p_{k-1}\right\}} \mathcal{F}(p_1, p_2, \ldots, p_k)\\
= \sum_{p_1}\left(\sum_{p_2 }-\sum_{p_1 = p_2} \right)\cdots\left(\sum_{p_k}-\sum_{p_k = p_1}-\sum_{p_k = p_2}-\cdots-\sum_{p_k = p_{k-1}}\right) \mathcal{F}(p_1, p_2, \ldots, p_k).
\end{multline*}

Let us expand the parentheses in this equality. 
Then we get that $F$ is the sum of a number of terms of the form $\sum \sum \cdots \sum$ ($k$ sums) with a $\pm$ sign in front. 
For any such term,
we denote by $\mathcal{A}_1$ the set of all indices 
$\alpha_1\in[k]$ for which $p_{\alpha_1}=p_1$ 
(taking into account all the relevant equalities), and write $q_1$ for this number. Next, let $q_2$ be 
the first number among $p_1, p_2,\ldots$ not yet occurring among the $p_\alpha$'s with $\alpha\in\mathcal{A}_1$. Denote 
by $\mathcal{A}_2$ the set of those indices $\alpha_2$ for which $p_{\alpha_2}=q_2$. Proceeding 
in this way, we obtain a partition of the index set $[k]$:
$$
  \mathcal{A}_1 \mathbin{\scalebox{1.4}{$\sqcup$}} 
  \mathcal{A}_2 \mathbin{\scalebox{1.4}{$\sqcup$}} \cdots \mathbin{\scalebox{1.4}{$\sqcup$}}\mathcal{A}_{\ell} \;=\; [k],
$$
where $\ell$ is a number of $q$'s, and $i<j$ implies $\min \mathcal{A}_i<\min \mathcal{A}_j$.

For each $1\leqslant v \leqslant \ell$ we let
 $$\mathcal{A}_\nu = \left\{\alpha_{1}<\alpha_{2}<\cdots<\alpha_{r}\right\},\ \ r = \# \mathcal{A}_\nu.$$
We also set
$$P_{\mathcal{A}_\nu} = \sum_{ p_{\alpha_{2} } = p_{\alpha_{1}}}\sum\limits_{\substack{p_{\alpha_{ 3}} \in \left\{p_{\alpha_{1}}, p_{\alpha_{ 2}}\right\}}}\cdots \sum\limits_{\substack{p_{\alpha_{ r}} \in \left\{p_{\alpha_{ 1}}, p_{\alpha_{2}}, \ldots, p_{\alpha_{r-1}}\right\}}}(-1)^{r-1}.$$

Then expanding  the parentheses, we get
$$F = \sum_{\ell = 1}^k \sum\limits_{\substack{\mathcal{A}_1 \sqcup  \cdots \sqcup \mathcal{A}_\ell = [k]\\ 
i < j \Rightarrow \min A_i < \min A_j}} \sum\limits_{\substack{q_1,\ldots, q_\ell \in \mathcal{P} \\
p_j = q_\nu \forall j \in \mathcal{A}_\nu}} 
P_{\mathcal{A}_1} P_{\mathcal{A}_2}\cdots P_{\mathcal{A}_\ell}\cdot \mathcal{F}(p_1, \ldots, p_k).
$$

Since 
$$P_{{\mathcal{A}_\nu}} = (-1)^{\#\mathcal{A}_\nu - 1}(\#\mathcal{A}_\nu -1 )!$$
and
$$(\#\mathcal{A}_1-1)+\cdots+(\#\mathcal{A}_\ell-1) = k-\ell,$$ the claim follows.
\end{proof}

If we apply the above theorem to a tree zeta function we get an expression in terms of prime zeta functions. This will allow us to get an analytic continuation of the tree zeta function in a domain larger that $\sigma > 1$.
\begin{prop} 
Let $T=e^{T_1} \circ e^{T_2} \circ \cdots \circ e^{T_k}$, where ${\mathfrak M}(T_i) \in {\mathbb N}$ for all $i$. 
Then for $\sigma = \Re s > 1$ we have
\begin{equation}\label{treezeta}
\zeta_T(s) = \sum_{\ell=1}^k (-1)^{k - \ell} 
\sum\limits_{\substack{A_1 \sqcup \cdots \sqcup A_\ell = [k]\\ 
i < j \Rightarrow \min A_i < \min A_j}}c_{\mathcal{A}_1, \ldots, \mathcal{A}_{\ell}}
\sum\limits_{\substack{v_1,\ldots,v_k\\ 
t(v_i) = T_i \ (1 \leq i \leq k)}}
P\Big(s \sum_{\alpha_1 \in A_1}v_{\alpha_1}\Big) \cdots
P\Big(s \sum_{\alpha_\ell \in A_\ell}v_{\alpha_\ell}\Big),
\end{equation}
where
$$c_{\mathcal{A}_1, \ldots, \mathcal{A}_{\ell}} = (\#\mathcal{A}_1-1)!\cdots (\#\mathcal{A}_\ell-1)!.$$
\end{prop}	
\begin{proof} We have
$$\zeta_T(s) = \sum\limits_{\substack{p_1, \ldots, p_k  \\ i \neq j \Rightarrow p_i \neq p_j }}\sum\limits_{\substack{v_1, \ldots, v_k \\ t(v_i) = T_i}}\left(p_1^{v_1}\cdots p_k^{v_k}\right)^{-s}.$$
Using formula \eqref{sumformula}, we get
$$\zeta_T(s) = \sum_{\ell = 1}^k (-1)^{k-\ell} \sum\limits_{\substack{A_1 \sqcup \cdots \sqcup A_\ell = [k]\\ 
i < j \Rightarrow \min A_i < \min A_j}}c_{\mathcal{A}_1, \ldots, \mathcal{A}_\ell}\sum_{q_1, \ldots, q_\ell 
}\sum\limits_{\substack{v_1, \ldots, v_k \\ t(v_i) = T_i}}\left(q_1^{-s\sum_{\alpha_1\in \mathcal{A}_1}v_{\alpha_1}}\right)\cdots \left(q_\ell^{-s\sum_{\alpha_\ell\in \mathcal{A}_\ell}v_{\alpha_\ell}}\right),$$
where $c_{\mathcal{A}_1, \ldots, \mathcal{A}_\ell}$ is defined in the assumption.
Changing the order of summation gives the result.

\end{proof}

The identity \ref{treezeta} can be analytically continued to a suitable subset of the region $\sigma >0$ that does not contain the point $s=1/m$, where
$m = \min \{\mathfrak{M}(T_i): 1 \leq i \leq k\}$.

\medskip
Using the previous proposition, one can prove the following result.
\begin{theorem}\label{main2}
Let $T = \underbrace{ {e^{T_0}\circ e^{T_0}\circ \cdots \circ e^{T_0}}}_{k\ \textup{times}}\circ\,e^{T_{k+1}} \circ \cdots \circ e^{T_K}$,
with $m:= {\mathfrak M}(T_0) < {\mathfrak M}(T_i)$, for all $k < i \leq K$.
Then, for $s \to 1/m$, $\Re s> 1/m$, we have
\begin{equation}
\zeta_T(s) \sim \zeta_{T'}\Big(\frac 1 m\Big) \Big(\log\Big(\dfrac{1}{s-1/m}\Big)\Big)^k
\end{equation}
where $T' = e^{T_{k+1}} \circ \cdots \circ e^{T_K}$.
\end{theorem}

\begin{proof}

Taking the logarithm of the Euler product expansion for the Riemann zeta function, we obtain  
\[
\log \zeta(s) = -\sum_{p}\log\!\left(1-\frac{1}{p^s}\right), \qquad \Re s > 1,
\]  
where the principal branch of the logarithm is chosen. Expanding the logarithm into a Taylor series, we get  
\[
\log\zeta(s) = P(s) + Q(s), \qquad 
Q(s) = \sum_{p}\left(-\frac{1}{p^s}-\log\!\left(1-\frac{1}{p^s}\right)\right) 
= \sum_{p}\sum_{k\geqslant 2}\frac{1}{k\,p^{ks}}.
\]

The series defining $Q(s)$ converges absolutely and uniformly for $\Re s > \tfrac{1}{2}$. Hence $Q(s)$ is analytic in the region  
\[
\mathcal{G} = \{\, s\neq 1 : \Re s > \tfrac{1}{2}\,\}.
\]  
In particular, $Q(s)$ is bounded as $s \to 1$. Therefore, as $s \to 1$ with $\Re s > 1$,
\begin{equation}\label{logzeta}
\log\zeta(s) = P(s) + O(1).
\end{equation}

Finally, using the classical expansion  
\[
\zeta(s) = \frac{1}{s-1} + O(1), \quad s \to 1,
\]  
we deduce  
\begin{equation}\label{P(s)}
P(s) = \log \frac{1}{s-1} + O(1), \qquad s \to 1,\ \ \Re s>1.
\end{equation}

Consider now equality \eqref{treezeta} and denote by $Z_1$ the contribution coming from the partitions of $[K]$ of the form
\begin{equation}\label{cond1}
\mathcal{A}_i = \left\{i\right\}\  \text{for}\ \ 1\leqslant i\leqslant k,
\end{equation}
\begin{equation}\label{cond2}
\mathcal{A}_{k+1}\scalebox{1.5}{$\sqcup$}\cdots \scalebox{1.5}{$\sqcup$} \mathcal{A}_{\ell} = [K]\setminus[k],
\end{equation}
and let $Z_2 = \zeta_T(s)-Z_1.$ 

First, let us evaluate the sum $Z_1$.  If $K = k$, then $T' = t(1)$ and
$$Z_1 = \left(\log\dfrac{1}{s-\frac{1}{m}}\right)^k  = \left(\log\dfrac{1}{s-\frac{1}{m}}\right)^k\zeta_{T'}(s).$$ Suppose now that $K>k$. Then since the term corresponding to $\ell \leqslant k$ is equal to zero, we have
\begin{multline}
Z_1 = \sum_{\ell = k+1}^K (-1)^{K-\ell} \sum\limits_{\substack{A_{k+1} \sqcup \cdots \sqcup A_\ell = [K]\setminus [k]\\ 
i < j \Rightarrow \min A_i < \min A_j}}c_{\mathcal{A}_{1}, \ldots, \mathcal{A}_\ell}\\
\times\left(\sum_{t(v) = T_0}P(sv)\right)^k\sum\limits_{\substack{v_{k+1}, \ldots, v_K \\ t(v_i) = T_i}}P\left(s\sum_{\alpha_{k+1}\in \mathcal{A}_{k+1}}v_{\alpha_{k+1}}\right)\cdots P\left(s\sum_{\alpha_{\ell}\in \mathcal{A}_{\ell}}v_{\alpha_\ell}\right)\!. 
\end{multline}
Since $P(s)\ll 2^{-\sigma} $ for $\sigma > 1$, it follows from the formula for the sum of geometric progression, that
\begin{equation}\sum_{t(v) = T_0}P(sv) = P(sm) + O\left(\sum_{v\geqslant m+1} 2^{-\sigma v}\right) = P(sm) + O\left(\dfrac{1}{2^{\sigma(m+1)}-1}\right).
\end{equation}

Then, as $s\to 1/m$, $\Re s>1/m$, it follows from \eqref{P(s)} that
\begin{equation}\sum_{t(v) = T_0}P(sv) 
=\log\dfrac{1}{s-\frac{1}{m}}+O_m(1),
\end{equation}
and then, by the binomial theorem,
$$\left(\sum_{t(v) = T_0}P(sv)\right)^k  = \left(\log\dfrac{1}{s-\frac{1}{m}}\right)^k + O_{m, k}\left(\left|\log\dfrac{1}{s-\frac{1}{m}}\right|\right)^{k-1}.$$
Let us compute the contribution $Z_1^{(1)}$ of the term $\log(1/(s-1/m))$ to the sum $Z_1$. We have
\begin{multline}
Z_1^{(1) } = \left(\log\dfrac{1}{s-\frac{1}{m}}\right)^k \sum_{\ell = k+1}^K (-1)^{K-\ell} \sum\limits_{\substack{A_{k+1} \sqcup \cdots \sqcup A_\ell = [K]\setminus [k]\\ 
i < j \Rightarrow \min A_i < \min A_j}}c_{\mathcal{A}_{1}, \ldots, \mathcal{A}_\ell}\\
\times\sum\limits_{\substack{v_{k+1}, \ldots, v_K \\ t(v_i) = T_i}}P\left(s\sum_{\alpha_{k+1}\in \mathcal{A}_{k+1}}v_{\alpha_{k+1}}\right)\cdots P\left(s\sum_{\alpha_{\ell}\in \mathcal{A}_{\ell}}v_{\alpha_\ell}\right)\!.
\end{multline}
Consider the bijection $\psi: [K]\setminus[k]\to [K-k]$ defined by the rule
$\psi(a) = a-k$ for each $k<a\leqslant K.$
For $k+1 \leq \ell \leq K$ and a partition
$A_{k+1}, \ldots, A_\ell$ of $[K] \setminus [k]$,
let $\tilde{A}_{\nu-k}= \psi(A_\nu)$ for $k+1 \leq \nu\leqslant \ell$, 
so that $\tilde{A}_1,\ldots, \tilde{A}_{\ell'}$ is a partition of $[K - k]$,
$\ell' = \ell-k$. 
Moreover, if $i < j$ implies $\min A_i < \min A_j$, then it also implies $\min \psi(A_i) < \min \psi(A_j)$.
Also, set $\tilde{v}_{i-k} =v_i $ for $k<i\leqslant K$. 
Since the map $\psi$ sets up a bijection between partitions of the set $[K]\setminus [k]$ and partitions of the set $[K - k]$, and since $$c_{\mathcal{A}_1, \ldots, \mathcal{A}_\ell} = (\#\mathcal{A}_{k+1}-1)!\cdots (\#\mathcal{A}_{\ell}-1)! = (\#\tilde{\mathcal{A}}_{1}-1)!\cdots (\#\tilde{\mathcal{A}}_{\ell'}-1)! = c_{\tilde{\mathcal{A}}_1, \ldots, \tilde{\mathcal{A}}_{\ell'}},$$
it follows that
\begin{multline}
    Z_1^{(1) } = \left(\log\dfrac{1}{s-\frac{1}{m}}\right)^k \sum_{\ell' = 1}^{K-k} (-1)^{K-k-\ell'} \sum\limits_{\substack{\tilde{\mathcal{A}}_1 \sqcup \cdots \sqcup \tilde{\mathcal{A}}_{\ell'}  = [K-k]\\ 
i < j \Rightarrow \min \tilde{\mathcal{A}}_i  < \min \tilde{\mathcal{A}}_j }}c_{\tilde{\mathcal{A}}_1 , \ldots, \tilde{\mathcal{A}}_{\ell'} }\\
\times\sum\limits_{\substack{\tilde{v}_{1}, \ldots, \tilde{v}_{K-k} \\ t(\tilde{v}_i) = T_{i+k}}}P\left(s\sum_{\alpha_{1}\in \tilde{\mathcal{A}}_1 }\tilde{v}_{\alpha_{1}}\right)\cdots P\left(s\sum_{\alpha_{\ell}\in \tilde{\mathcal{A}}_{\ell'}}\tilde{v}_{\alpha_{\ell'}}\right)\!
= \left(\log\dfrac{1}{s-\frac{1}{m}}\right)^k\zeta_{T'}(s).
\end{multline}
Now we are going to estimate the contribution
$Z_1^{(2)}$ of the term $O(|\log(1/(sm-1))|^{k-1})$ to the sum $Z_1$. 
For a finite set $\mathcal{A}$ of positive integers, consider the sum
$$\mathcal{P}_{\mathcal{A}}(s) = \sum\limits_{\substack{v_\alpha, \alpha \in \mathcal{A} \\ t(v_\alpha) = T_{\alpha} \ (k<\alpha\leqslant K) \\ t(v_\alpha) = T_0\ (1\leqslant \alpha\leqslant k)}} P\left(s\sum_{\beta \in \mathcal{A}} v_\beta\right)\!.$$
We have
$\mathcal{P}_{\mathcal{A}}(s) = P_1+P_2,$
where
$$P_1 = P(sm)\mathbb{I}_{\left(\mathcal{A} = \left\{j\right\} \ \text{for some}\ 1\leqslant j \leqslant k \right)}$$
and
$$P_2 = \sum_{n\geqslant m+1} P(sn) q(n),\quad q( n) = \sum\limits_{\substack{\alpha \in \mathcal{A} \\ t(v_\alpha) = T_{\alpha} \ (k<\alpha\leqslant K) \\ t(v_\alpha) = T_0\ (1\leqslant \alpha\leqslant k) \\ \sum_{\beta \in\mathcal{A}}v_\beta = n}} 1.$$
Set $a = \#\mathcal{A}$. Then since
$$q(n)\leqslant \#\left\{(x_1, \ldots, x_{a}) \in \mathbb{N}^a: x_1+\cdots+x_{a} = n\right\} \leqslant n^{a},$$ it follows that
$$|P_2|\ll \sum_{n\geqslant m+1}2^{-\sigma n} n^a\ll_a \sum_{n\geqslant m+1}(\sqrt{2})^{-\sigma n}\ll_m 1,$$
as $s \to {1}/{m}$, $\Re s>1/m$. Hence, using the equality
$$\sum\limits_{\substack{v_{k+1}, \ldots, v_K \\ t(v_i) = T_i}}P\left(s\sum_{\alpha_{k+1}\in \mathcal{A}_{k+1}}v_{\alpha_{k+1}}\right)\cdots P\left(s\sum_{\alpha_{\ell}\in \mathcal{A}_{\ell}}v_{\alpha_\ell}\right) = \prod_{\nu = k+1}^\ell \mathcal{P}_{\mathcal{A}_\nu}(s),$$ 
and the fact that $\mathcal{A}_\nu \neq \left\{j\right\}$ for all $1\leqslant j\leqslant k$ and $k+1\leqslant \nu\leqslant \ell$,
we get
$$Z_1^{(2)}\ll_m \left|\log\dfrac{1}{s-\frac{1}{m}}\right|^{k-1}.$$
Thus,
$$Z_1 = \left(\log\dfrac{1}{s-\frac{1}{m}}\right)^{k}+O_{m, k}\left(\left|\log\dfrac{1}{s-\frac{1}{m}}\right|^{k-1}\right)\!$$
as $s\to {1}/{m}$, $\Re s>1/m$.

Finally, it remains to estimate the value $Z_2$. To do this, we note that a violation of conditions \eqref{cond1} and \eqref{cond2} implies that $\#\mathcal{A}_j\geqslant 2$ for some $1\leqslant j\leqslant k$. Indeed, assume the converse. Then for all $1\leqslant j\leqslant k$ we have $\#\mathcal{A}_j=1$ and there is $j_0\geqslant 2$ such that
$$j_0 = \min\left\{1<j\leqslant k: \mathcal{A}_j = \left\{i\right\}, i\neq j\right\}.$$
Since $\min \mathcal{A}_i\geqslant i$, it follows that
$$j_0 = \min\left\{1<j\leqslant k: \mathcal{A}_j = \left\{i\right\}, i> j\right\}.$$
Then on the one hand, 
$$j_0 \not{\in} 
  \mathcal{A}_1 \mathbin{\scalebox{1.4}{$\sqcup$}} 
  \mathcal{A}_2 \mathbin{\scalebox{1.4}{$\sqcup$}} \cdots \mathbin{\scalebox{1.4}{$\sqcup$}}\mathcal{A}_{j_0-1} = \left\{1, 2, \ldots, j_0-1\right\}.
$$
On the other hand, since
$$\min \mathcal{A}_\ell>\min\mathcal{A}_{\ell-1}>\cdots>\min\mathcal{A}_{j_0}>j_0,$$
it follows that 
$$j_0 \not{\in} 
  \mathcal{A}_{j_0} \mathbin{\scalebox{1.4}{$\sqcup$}} 
  \mathcal{A}_{j_0+1} \mathbin{\scalebox{1.4}{$\sqcup$}} \cdots \mathbin{\scalebox{1.4}{$\sqcup$}}\mathcal{A}_{\ell}$$
  and
  $$j_0\not{\in} \mathcal{A}_{1} \mathbin{\scalebox{1.4}{$\sqcup$}} 
  \mathcal{A}_{2} \mathbin{\scalebox{1.4}{$\sqcup$}} \cdots \mathbin{\scalebox{1.4}{$\sqcup$}}\mathcal{A}_{\ell} = [K]\supseteq[k].$$ This is a contradiction.

  Hence, for at least one $1\leqslant j\leqslant k$ we have $\mathcal{A}_j \neq \left\{i\right\}$ for all $1\leqslant i\leqslant k$. Therefore, we find that
  $$\prod_{\nu = 1}^\ell \mathcal{P}_{\mathcal{A}_\nu}(s)\ll_m \left|\prod_{j = 1}^k \mathcal{P}_{\mathcal{A}_j}(s)\right|\ll_m \left|\log\dfrac{1}{s-\frac{1}{m}}\right|^{k-1}.$$
  Thus, 
  $$Z_2 = \sum_{\ell=1}^K (-1)^{K
  - \ell} 
\sum\limits_{\substack{A_1 \sqcup \cdots \sqcup A_\ell = [K]\\ 
i < j \Rightarrow \min A_i < \min A_j \\ \exists j\ (1\leqslant j\leqslant k)\  \#\mathcal{A}_j\geqslant 2}}c_{\mathcal{A}_1, \ldots, \mathcal{A}_{\ell}}\prod_{\nu = 1}^\ell \mathcal{P}_{\mathcal{A}_\nu}(s)\ll_m \left|\log\dfrac{1}{s-\frac{1}{m}}\right|^{k-1}$$
and
$$\zeta_T(s) = \left(\log\dfrac{1}{s-\frac{1}{m}}\right)^{k}\zeta_{T'}(s)+O_m\left(\left|\log\dfrac{1}{s-\frac{1}{m}}\right|^{k-1}\right)\sim \left(\log\dfrac{1}{s-\frac{1}{m}}\right)^{k}\zeta_{T'}\left(\dfrac{1}{m}\right),$$
as $s \to {1}/{m}$, $\Re s>1/m$. The proof of the theorem is complete.
\end{proof}

It follows from the relation 
$$P(s) = \sum_{n=1}^{+\infty}\dfrac{\mu(n)}{n}\log \zeta(sn),$$
where $\mu(n)$ denotes the M\"obius function, and $\sigma >0, s\neq 1/n, s\neq \rho/n$ for each $n\geqslant 1$ and each non-trivial zero $\rho$ of $\zeta$ (see \cite[Chapter 1, \S 6]{Tit},
identity 1.6.1) that
$P(s)$ has essential singularities at the points $s = 1/n$ and $s = \rho/n$ for $n\geq 1$. It seems interesting to describe all the singular points of $\zeta_T(s)$ and find the asymptotic behavior of $\zeta_T(s)$ near them.

\end{document}